\documentclass{amsart}

\usepackage{amsmath}
\usepackage{amssymb}
\usepackage{amscd}
\usepackage[all]{xy}
\usepackage{hyperref}
\usepackage{wasysym}

\input{counting.sty}

\newcommand{\op}[1]{\operatorname{#1}}

\begin{document}
\title{Counting Using Hall Algebras III. Quivers with Potentials}

\author{Jiarui Fei}
\address{Shanghai Jiao Tong University, School of Mathematical Sciences}
\email{jiarui@sjtu.edu.cn}
\address{University of California, Riverside, Department of Mathematics}
\thanks{}


\date{}
\dedicatory{}

\begin{abstract} For a quiver with potential, we can associate a vanishing cycle to each representation space.
If there is a nice torus action on the potential, the vanishing cycles can be expressed in terms of truncated Jacobian algebras.
We study how these vanishing cycles change under the mutation of Derksen-Weyman-Zelevinsky.
The wall-crossing formula leads to a categorification of quantum cluster algebras under some assumption.
This is a special case of A. Efimov's result, but our approach is more concrete and down-to-earth.
We also obtain a counting formula relating the representation Grassmannians under sink-source reflections.
\end{abstract}

\maketitle

\section*{Introduction}

We continue our development on algorithms to count the points of varieties related to quiver representations. In this note, we focus on the quivers with potentials. A {\em potential} $W$ on a quiver $Q$ is just a linear combination of oriented cycles of $Q$. It can be viewed as a function on a certain noncommutative space attached to $Q$. When composed with the usual trace function, it becomes a regular function $\omega$ on each representation space $\Rep_\alpha(Q)$. This function further descends to various moduli spaces.
In this paper, all potentials are assumed to be polynomial, i.e., a finite linear combination of oriented cycles.

Let $f$ be a regular function on a complex variety $X$. Consider the scheme theoretic critical locus $\{df=0\}$ of $f$.
Behrend, Bryan and Szendr\"{o}i define in \cite{BBS} a class $[\varphi_f(X)]\in K_0(\op{Var}_{\mb{C}})[\mb{L}^{-\frac{1}{2}}]$ associated to each
such locus, essentially given by the motivic {\em Milnor fibre} of the map $f$.
When $X$ admits a suitable torus action, this class can be expressed as\footnote{This definition of $[\varphi_f(X)]$ differs from the original one by a negative sign.}
$$[\varphi_f(X)] = [f^{-1}(0)] - [f^{-1}(1)].$$

Deligne's mixed Hodge structure on compactly supported cohomology gives rise to the $E$-polynomial homomorphism $E: K_0(\op{Var}_{\mb{C}})\to \mb{Z}[x,y]$ given by
$$E([Y];x,y)=\sum _{p,q} x^p y^q \sum_i (-1)^i\dim H_{p,q}(H^i_c(Y,\mb{Q})).$$
The $E$-polynomial could sometimes be computed using arithmetic method.
By a {\em spreading out} of a complex variety $Y$, we mean a separated scheme $\mathcal{Y}$ over a finitely generated
$\mb{Z}$-algebra $R$ with an embedding $\varphi: R\hookrightarrow \mb{C}$ such that the extension of scalars $\mathcal{Y}_\varphi \cong Y$.
Following N. Katz \cite[Appendix]{HR}, we say that $Y$ is polynomial-count if there is a polynomial $P_Y(t)\in\mb{Z}[t]$ and a spreading out $\mathcal{Y}$ such that for every homomorphism $\varphi: R\to \mb{F}_q$ to a finite field, the number of $\mb{F}_q$-points of the scheme $\mathcal{Y}_\varphi$ is $P_Y(q)$. Furthermore, the definitions descend to the Grothendieck group $K_0(\op{Var}_{\mb{C}})$.
It is known (\cite[Theorem 6.1.2, 6.1.3]{HR}) that
\begin{lemma} \label{L:polycount}
Assume that $\gamma \in K_0(\op{Var}_{\mb{C}})$ is {\em polynomial-count} with counting polynomial
$P_\gamma(t)\in \mb{Z}[t]$, then the $E$-polynomial of $\gamma$ is given by
$$E(\gamma; x, y) = P_\gamma(xy).$$
\end{lemma}


In this note, we directly work over finite fields $\mb{F}_q$ and follow an algebraic approach to compute $P_\gamma(t)$ in the above quiver setting. Namely, $\gamma$ is the class $[\varphi_\omega]$, defined by the regular function $\omega$ on representation spaces.
In fact, we will work with the generating series of the {\em virtual point counts} $|\varphi_\omega(X)|_\vir$ for all dimension vectors:
$$\mb{V}(Q,W):=\sum_\alpha \frac{|\varphi_\omega(\Rep_\alpha(Q))|_\vir}{|\GL_\alpha|_\vir} x^\alpha.$$
Here, by definition the virtual point count is related to the ordinary point count by a $q$-shift: $|\varphi_\omega(X)|_\vir:=q^{-\frac{\dim X}{2}}|\varphi_\omega(X)|$.
Our main results contain two wall-crossing formulas -- one for the ordinary $\mb{V}(Q,W)$, the other for the one with a {\em framing stability} $\nu_\infty$:
$$\T(Q,W):=\sum_\beta \left|\varphi_\omega(\Mod_{(1,\beta)}^{\nu_\infty}(Q))\right|_\vir\ x^{(1,\beta)},$$
where $\Mod_\alpha^\nu(Q)$ denotes the GIT moduli of $\alpha$-dimensional $\nu$-stable representations.
To state them, let $\mu_k$ be the mutation of $(Q,W)$ in the sense of Derksen, Weyman, and Zelevinsky \cite{DWZ1}. Let
$\mb{E}_k:=\exp_q(\frac{q^{1/2}}{q-1}x_k),$
$\mb{V}:=\mb{V}(Q,W), \mb{V}':=\mb{V}(\mu_k(Q,W))$ and $\T':=\T(\mu_k(Q,W))$, then
under the technical condition $\smiley$ of Section \ref{S:QPmu},
\begin{theorem} We have that \begin{align}
\label{Teq:Vmu} &\mb{E}_k'\Phi_k^\vee(\mb{V}\mb{E}_k^{-1})=\mb{V}'=\Phi_k(\mb{E}_k^{-1}\mb{V})\mb{E}_k'; \\
\label{Teq:Fmu} (&\mb{E}_k')^{-1}{\Phi}_k(\T) \mb{E}_k'=
\T' = {\Phi}_k^\vee (\mb{E}_k\T\mb{E}_k^{-1}).
\end{align}
Here, multiplications are performed in appropriate completed quantum Laurent series algebras. $\Phi_k^\vee$ and $\Phi_k$ are certain linear monomial change of variables.
\end{theorem}

There are two main ingredients in the proof. One is a construction of S. Mozgovoy, which relates the Hall algebra of a quiver to the quantum Laurent series ring. The other is a `dimension reduction' technique used by A. Morrison and K. Nagao. The equation \eqref{Teq:Vmu} already appeared in \cite{N2,Ke2}, 
but we put it in a right assumption
\footnote{In \cite{N2} the author only treats a rather special case when $k$ is a {\em strict sink/source}.
Moreover, the proof contains a gap due to the incorrectness of Lemma 4.3.
In \cite{Ke2} the author considered a variation of $\mb{V}$ and assumed an unproven conjecture (Conjecture 3.2). }.

Nagao also suggested in \cite{N2} that this theory can be used to study quantum cluster algebra. We follow his suggestion and use \eqref{Teq:Fmu} to categorify quantum cluster algebras under the assumption of the existence of certain potentials. If a cluster algebra has such a categorification, then its {\em strong positivity} will be implied by a result of \cite{DMSS} on the purity of the vanishing cycles.

Let $B$ be an $n\times m$ matrix with $n\leq m$ such that the left $n\times n$ submatrix of $B$ is skew-symmetric.
Let $\Lambda$ be another skew-symmetric matrices of size $m\times m$.  
We assume that $\Lambda$ and $B$ are {\em unitally compatible}, that is, ${B}\Lambda = (-I_n,0)$.
We can associate to $B$ a quiver $Q$ without loops and 2-cycles such that
$$b_{ij}=|\text{arrows }j\to i|-|\text{arrows }i\to j|.$$
Such a matrix is called the {\em $B$-matrix} of $Q$. We endow $Q$ with some potential $W$ having nice properties.

Let $\mu_{\bold{k}_t}$ be a sequence of {\em mutations}, and set $({Q}_t,W_t)=\mu_{\bold{k}_t}({Q},W)$.
Let $x^g\ (g\in \mb{Z}_{\geq 0}^m)$ be some initial {\em cluster monomial} in the quantum Laurent polynomial ring $X_\Lambda$ (see \eqref{eq:QLPA}). We extend QP $(Q_t^g,W_t)$ from $(Q_t,W_t)$ by adding a new vertex $\infty$ and $g_i$ new arrows from $i$ to $\infty$. We apply the inverse of $\mu_{\bold{k}_t}$ to $(Q_t^g,W_t)$, and obtain a QP $(Q^g,W^g):=\mu_{\bold{k}_t}^{-1}(Q_t^g,W_t)$. Let $\wtd{B}$ be the $B$-matrix of $Q^g$.
\begin{theorem} The {\em mutated} cluster monomial $X_t(g):=\mu_{\bold{k}_t}(x^g)$ is equal to
$$\sum_\beta \left|\varphi_{\omega^g}(\Mod_{(1,\beta)}^{\nu_\infty}(Q^g))\right|_\vir\ x^{(1,\beta)\wtd{B}}.$$
\end{theorem}
This result may just be a special case of \cite{E}, where A. Efimov assumed a much weaker condition on the potential $W$. However, our approach and result are more down-to-earth and computable. It depends only on \cite{BBS} rather than \cite{KS}.

The second application is on the generating series counting subrepresentations of representations of quiver (with zero potential).
Let $s$ be a sink of $Q$, and $M$ be a representation of $Q$. We assume that $M$ does not contain the simple representation $S_s$ as a direct summand. Let
$$\T(M):=\sum_\beta q^{-\frac{1}{2}\innerprod{\br{M}-\beta,\beta}_Q} {\big|\Gr^\beta(M)\big|} x^{(1,\beta)},$$
where $\Gr^\beta(M)$ is the variety parameterizing  $\beta$-dimensional quotient representations of $M$.
\begin{theorem} $\T(M)$ and $\T(\mu_sM)$ are also related via \eqref{Teq:Fmu}.
In particular, if $M$ is {\em polynomial-count}, that is, all its Grassmannians $\Gr^\beta(M)$ are polynomial-count, then so are all reflection equivalent classes of $M$.
\end{theorem}

This note is organized as follows. In Section \ref{S:def}, we recall some basics about quiver representations and their Hall algebras.
In Section \ref{S:QPC}, we recall the concept of quiver with potential and a cut, and its associated algebras.
In Section \ref{S:Mozgovoy}, we recall a construction of Mozgovoy and the dimension reduction technique (Lemma \ref{L:VtoR}).
In Section \ref{S:QPmu}, we recall the mutation of QP with a cut, and set up the key assumption for our main results. When this assumption holds is illustrated in Corollary \ref{C:tilting}, whose proof will be given in the appendix.
In Section \ref{S:WC}, we prove our first two main results -- Theorem \ref{T:WC} and Theorem \ref{T:FWC}.
In Section \ref{S:cluster}, we prove our third main result, Theorem \ref{T:cluster}, on a categorification of quantum cluster algebras.
In Section \ref{S:QGrass}, we prove our last main result, Theorem \ref{T:QGrass}, on the representation Grassmannians under reflections.

\subsection*{Notations and Conventions}
\begin{enumerate}
\item[$\bullet$] All modules are right modules and all vectors are row vectors. 
\item[$\bullet$] For an arrow $a$, $ta$ is the tail of $a$ and $ha$ is the head of $a$.
\item[$\bullet$] For any representation $M$, we use $\br{M}$ to denote its dimension vector.
\item[$\bullet$] $S_i$ is the simple module at the vertex $i$, and $P_i$ is its projective cover.
\item[$\bullet$] Superscript $*$ is the trivial dual $\Hom_K(-,K)$.
\end{enumerate}

\section{Basics on Quivers and their Hall algebras} \label{S:def}

From now on, we assume our base field $K=\mb{F}_q$ to be the finite field with $q$ elements. 
Let $Q$ be a finite quiver with the set of vertices $Q_0$ and the set of arrows $Q_1$.
We write $\innerprod{-,-}_Q$ for the usual Euler form of $Q$:
$$\innerprod{\alpha,\beta}_Q=\sum_{i\in Q_0} \alpha(i)\beta(i) - \sum_{a\in Q_1} \alpha(ta)\beta(ha)\quad \text{for $\alpha,\beta\in \mbox{Z}^{Q_0}$}$$ 
We also have the antisymmetric form $(-,-)$ associated to $Q$. The matrix of $(-,-)$ denoted by $B$ is given by
\begin{equation} \label{eq:bij} b_{ij}=|\text{arrows }j\to i|-|\text{arrows }i\to j|. \end{equation}

Let $KQ$ be the path algebra of $Q$ over $K$. For any three $KQ$-modules $U,V$ and $W$ with dimension vector $\beta,\gamma$ and $\alpha=\beta+\gamma$,
the {\em Hall number} $F_{UV}^W$ counts the number of subrepresentations $S$ of $W$ such that $S\cong V$ and $W/S\cong U$.
We denote $a_W:=|\Aut_Q(W)|$, where $|X|$ is the number of $K$-rational points of $X$.
Let $H(Q)$ be the vector space of all formal (infinite) linear combinations of isomorphism classes of $KQ$-modules.
\begin{lemma}[{\cite{R},\cite[Proposition 1.1]{S}}]
The completed {\em Hall algebra} $H(Q)$ is the associative algebra with multiplication
$$[U][V]:=\sum_{[W]}F_{UV}^W[W],$$ and unit $[0]$.
\end{lemma}

Let $\module KQ$ (resp. $\module_\alpha KQ$) be the category of all finite dimensional (resp. $\alpha$-dimensional) $KQ$-modules. 
For a subcategory $\mc{C}$ of $\module KQ$, we denote $\chi(\mc{C}):=\sum_{M\in\mc{C}} [M]$. We use the shorthand $\chi$ and $\chi_\alpha$ for $\chi(\module KQ)$ and $\chi(\module_\alpha KQ)$. Let $(\mc{T},\mc{F})$ be a {\em torsion pair} (\cite[Definition VI.1.1]{ASS}) in $\module KQ$, then for any $M\in\module KQ$, there exists a short exact sequence $0\to L\to M\to N\to 0$ with $L$ unique in $\mc{T}$ and $N$ unique in $\mc{F}$.
In terms of the Hall algebra, this says that
\begin{lemma} \label{L:torsionid} $\chi=\chi(\mc{F})\chi(\mc{T}).$
\end{lemma}

A {\em weight} $\sigma$ is an integral linear functional on $\mathbb{Z}^{Q_0}$. A {\em slope function} $\nu$ is a quotient of two weights $\sigma/\theta$ with $\theta(\alpha)>0$ for any non-zero dimension vector $\alpha$.

\begin{definition} A representation $M$ is called {\em $\nu$-semi-stable (resp. $\nu$-stable)} if
$\nu(\br{L})\leqslant \nu(\br{M})$ (resp. $\nu(\br{L})<\nu(\br{M})$) for every non-trivial subrepresentation $L\subset M$.
\end{definition}

We denote by $\Rep_\alpha^\nu(Q)$ the variety of $\alpha$-dimensional $\nu$-semistable representations of $Q$.
By the standard GIT construction \cite{Ki}, there is a {\em categorical quotient} $q: \Rep_\alpha^\nu(Q)\to \Mod_\alpha^\nu(Q)$ and its restriction to the stable representations $\Rep_\alpha^{\nu\cdot{\rm st}}(Q)$ is a {\em geometric quotient}.

A slope function $\nu$ is called {\em coprime} to $\alpha$ if $\nu(\gamma)\neq \nu(\alpha)$ for any $\gamma<\alpha$. So if $\nu$ is coprime to $\alpha$, then there is no strictly semistable (semistable but not stable) representation of dimension $\alpha$. In this case, $\Mod_\alpha^\nu(A)$ must be a geometric quotient.

\begin{lemma}[Harder-Narasimhan filtration]
Every representation $M$ has a unique filtration $$0=M_0\subset M_1\subset\cdots\subset M_{m-1}\subset M_{m}=M$$ such that $N_i=M_i/M_{i-1}$ is $\nu$-semi-stable and $\nu(\br{N}_i)>\nu(\br{N}_{i+1})$.
\end{lemma}

We fix a slope function $\nu$. For a dimension vector $\alpha$, let $\chi_\alpha^{\nu}:=\sum_{M\in\Rep_\alpha^{\nu}(Q)}[M]$.
The existence of the Harder-Narasimhan filtration yields the following identity in the Hall algebra $H(Q)$.
\begin{lemma}[{\cite[Proposition 4.8]{R1}}] \label{L:HNid} $$\chi_\alpha=\sum_{} \chi_{\alpha_1}^{\nu}\cdot\dots\cdot\chi_{\alpha_s}^{\nu},$$ where the sum runs over all decomposition $\alpha_1+\cdots+\alpha_s=\alpha$ of $\alpha$ into non-zero dimension vectors such that $\nu(\alpha_1)<\cdots<\nu(\alpha_s)$. In particular, solving recursively for $\chi_\alpha^{\nu}$, we get
\begin{equation}\label{eq:HallID} \chi_\alpha^{\nu}=\sum_* (-1)^{s-1}\chi_{\alpha_1}\cdot\cdots\chi_{\alpha_s},
\end{equation}
where the sum runs over all decomposition $\alpha_1+\cdots+\alpha_s=\alpha$ of $\alpha$ into non-zero dimension vectors such that $\nu(\sum_{l=1}^k\alpha_l)<\nu(\alpha)$ for $k<s$.
\end{lemma}

For $\alpha\in \mb{Z}^{Q_0}$, we write $x^\alpha$ for $\prod_{i\in Q_0} x_i^{\alpha(i)}$.
Let $X_{(Q)}$ be the quantum Laurent series algebra $\mb{Q}(q^{\frac{1}{2}})[\![x_i]\!]_{i\in Q_0}$ with multiplication given by $x^\beta x^\gamma=q^{-\frac{1}{2}(\beta,\gamma)}x^{\beta+\gamma}$.
\begin{lemma}[{\cite[Lemma 6.1]{R1}}] \label{L:int} 
The map $\int:[M]\mapsto q^{\frac{1}{2}\innerprod{\alpha,\alpha}_Q}\frac{x^\alpha}{a_M}$ is an algebra homomorphism $H(Q)\to X_{(Q)}$.
\end{lemma}
Since $\frac{1}{a_M}=\frac{|\GL_\alpha\cdot M|}{|\GL_\alpha|}$, we have that 
\begin{equation} \label{eq:int} \int \chi_\alpha=q^{\frac{1}{2}\innerprod{\alpha,\alpha}_Q}\frac{|\Rep_{\alpha}(Q)|}{|\GL_{\alpha}|}x^\alpha.
\end{equation}

\section{QP with a Cut} \label{S:QPC}
We fix a quiver $Q$ without loops or oriented 2-cycles and a potential $W$ on $Q$. 
Let $\widehat{KQ}$ be the completion of the path algebra $KQ$ with respect to the ${\mathfrak m}$-adic topology, where ${\mathfrak m}$ is the two-sided ideal generated by arrows of $Q$.
When dealing with completed path algebras, all its ideals will be assumed to be closed.
Recall that a {\em potential} $W$ is a linear combination of oriented cycles of $Q$.

For each arrow $a\in Q_1$, the {\em cyclic derivative} $\partial_a$ on $\widehat{KQ}$ is defined for each cycle $a_1\cdots a_d$ as
$$\partial_a(a_1\cdots a_d)=\sum_{a_i=a}a_{i+1}\cdots a_d a_1\cdots a_{i-1}.$$
For each potential $W$, its {\em Jacobian ideal} $\partial W$ is the (closed two-sided) ideal in $\widehat{KQ}$ generated by all $\partial_a W$. Let $J(Q,W)=\widehat{KQ}/\partial W$ be the {\em Jacobian algebra}. 
If $W$ is polynomial and $KQ/\partial W$ is finite dimensional, then the completion is unnecessary. This is the assumption we make throughout the paper.

Let $\omega: \Rep(Q)\to K$ be the {\em trace function} corresponding to the potential $W$ defined by $M\mapsto \tr(W(M))$.
Note that $\omega$ is an additive function, so it in fact descends to the Grothendieck group: $\omega: K_0(\module(Q))\to K$.
By abuse of notation, we also use $\omega$ for the same trace function defined on the affine representation varieties and moduli spaces. 
It is well-known that a representation $M$ of $Q$ is a representation of $J(Q,W)$ 
if and only if it is in the critical locus of $\omega$ (i.e., $d\omega(M) =0$).


Following \cite{HI}, we define
\begin{definition} A {\em cut} of a QP $(Q,W)$ is a subset $C\subset Q_1$ such that the potential $W$ is homogeneous of degree $1$ for the degree function $\deg: Q_1\to\mb{N}$ defined by $\deg(a)=1$ for $a\in C$ and zero otherwise.
\end{definition}

\noindent This degree function defines a $K^*$-action on $\Rep_\alpha(Q)$ 
\begin{equation}\label{eq:action} (tM)(i)=M(i)\text{ for } i\in Q_0,\quad (tM)(a) = t^{\deg(a)}M(a)\text{ for } a\in Q_1. \end{equation}
The homogenicity of $W$ implies that if $M$ is a representation of the Jacobian algebra, then so is $tM$.
Moreover, the trace function is equivariant: $\omega(tM)=t\omega(M)$.

\begin{definition} The algebra $J(Q,W;C)$ associated to a QP $(Q,W)$ with a cut $C$ is the quotient algebra of the Jacobian algebra $J(Q,W)$ by the ideal generated by $C$.
\end{definition}

\noindent For the degree function given by a cut, $\partial W$ is a homogeneous ideal so the degree function induces a grading on $J(Q,W)$ as well. Note that $J(Q,W;C)$ is isomorphic to degree zero part of $J(Q,W)$.
We denote by $Q_C$ the subquiver $(Q_0,Q_1\setminus C)$ of $Q$ and by $\innerprod{\partial_C W}$ the ideal $\innerprod{\partial_c W\mid c\in C}$. It is clear that $J(Q,W;C)$ can also be presented as $\widehat{KQ_C}/\innerprod{\partial_C W}$.
Readers may skip to Example \ref{ex:333} to see these definitions in action.

We put \begin{align*}
& \innerprod{\alpha,\beta}_C = \sum_{c\in C} \alpha(tc)\beta(hc),\\
& \innerprod{\alpha,\beta}_{J_C} = \innerprod{\alpha,\beta}_{Q}+\innerprod{\alpha,\beta}_{C}+\innerprod{\beta,\alpha}_{C}.
\end{align*}

\begin{definition}[\cite{HI}] \label{D:exact}
A cut is called {\em algebraic} if \begin{enumerate}
\item $J(Q,W;C)$ is a finite dimensional $K$-algebra of global dimension $2$;
\item $\{\partial_cW\}_{c\in C}$ is a minimal set of generators of the ideal $\innerprod{\partial_C W}$ in $\widehat{KQ}$.
\end{enumerate}
It is clear that for algebraic cuts, the form $\innerprod{-,-}_{J_C}$ is exactly the Euler form of $J(Q,W;C)$. From now on, we assume that all cuts are algebraic.
\end{definition}

Conversely, any finite-dimensional algebra of global dimension $2$ arises as an truncated Jacobian algebra (see \cite[Proposition 3.3]{HI}). 
Here is the construction. Given any $K$-algebra $A$ presented by a quiver $Q$ with a minimal set of relations $\{r_1,r_2,\dots,r_l\}$, we can associate with it a QP $(Q_A,W_A)$ with a cut $C$ as follows:
$Q_{A,0}=Q_0,\ Q_{A,1}=Q_1\amalg C$ with $C=\{c_i:h(r_i)\to t(r_i)\}_{i=1\dots l}$, and
$W_A=\sum_{i=1}^l c_ir_i$.
If $A$ has global dimension 2, it is known \cite[Theorem 6.10]{Ke1} that $J(Q_A,W_A)$ is isomorphic to the algebra
$\Pi_3(A):=\prod_{i\geq 0}\Ext_A^2(A^*,A)^{\otimes_A i}$. In particular, $J(Q_A,W_A)$ does not depend on the minimal set of relations that we chose.
It is now clear that $J(Q_A,W_A;C)\cong A$.
Moreover, let $C$ be an algebraic cut of $(Q,W)$ and $A=J(Q,W;C)$, then $(Q,W)$ and $(Q_A,W_A)$ are right-equivalent (\cite[Definition 4.2]{DWZ1}).
We do not need this construction in this paper though.

\section{A Construction of Mozgovoy} \label{S:Mozgovoy}

To motivate the definition \eqref{eq:phiw}, we consider a complex (quasiprojective) variety $X$ with a $\mb{C}^*$-action. 
Let $\omega$ be a regular function on $X$, which is equivariant with respect to a primitive character, i.e., not divisible in the character group of $\mb{C}^*$.
We further assume that $\lim_{t\to 0}tx$ exist for all $x\in X$, then according to \cite[Proposition 1.11]{BBS}, $[\omega^{-1}(1)]\in K_0(\op{Var}_{\mb{C}})$ is the {\em nearby fibre} of $\omega$.
Its difference with the {\em central fibre} $\omega^{-1}(0)$ defines a class called the {\em (absolute) vanishing cycle of} $\omega$ on $X$:
$$[\varphi_\omega(X)]:=[\omega^{-1}(0)]-[\omega^{-1}(1)].$$

Now let $X$ be a variety over $K=\mb{F}_q$ with a $K^*$-action, and $\omega$ still be a regular function on $X$. 
We set \begin{equation} \label{eq:phiw} |\varphi_\omega(X)|:=|\omega^{-1}(0)|-|\omega^{-1}(1)|. \end{equation}
Due to the torus action \eqref{eq:action}, we have that
\begin{align}(q-1)|\omega^{-1}(1)|&=|X|-|\omega^{-1}(0)|, \notag \\
\intertext{so}
\label{eq:virtual} |\varphi_\omega(X)| &=\frac{q|\omega^{-1}(0)|-|X|}{q-1}.
\end{align}
We denote the $q$-shifted point count $q^{-\frac{\dim X}{2}}|\varphi_\omega(X)|$ by $|\varphi_\omega(X)|_\vir$.
We also set $|\GL_\alpha|_\vir:=q^{-\frac{\dim\GL_\alpha}{2}}|\GL_\alpha|$.

Let $\omega: \Rep_\alpha(Q)\to K$ be the trace function corresponding to the potential $W$ with a cut. Recall that the cut induces a torus action on each $\Rep_\alpha(Q)$ such that $\omega$ is equivariant.
For $h=\sum c_M[M]\in H(Q)$, we define $h_0:=\sum_{\omega(M)=0} c_M[M]$.
Such an $h$ is called {\em equivariant} if $c_M=c_{tM}$ for any $t\in K^*$. Let $H_{\eq}(Q)$ be the subalgebra of $H(Q)$ consisting of equivariant elements.

\begin{lemma}[{\cite[Proposition 5.2]{M2}}] \label{L:intomega} The map $\int_{\omega}: H_{\eq}(Q)\to X_{(Q)}$ defined by
$$h\mapsto \frac{q\int h_0 - \int h}{q-1}$$ is an algebra morphism.
\end{lemma}

Note that if $W$ is zero, then $\int_\omega=\int$.
We see from \eqref{eq:virtual} and \eqref{eq:int} that
$$v_\alpha:=\int_{\omega} \chi_\alpha = \frac{|\varphi_\omega(\Rep_\alpha(Q))|_\vir}{|\GL_\alpha|_\vir}x^\alpha.$$
We denote the generating series $\int_{\omega}\chi$ by
$\mb{V}(Q,W):=\sum_\alpha v_\alpha x^\alpha.$

\begin{lemma}[{\cite[Theorem 4.1]{N2}}] \label{L:VtoR} $|\varphi_\omega(\Rep_\alpha(Q))|=q^{\innerprod{\alpha,\alpha}_C}|\Rep_\alpha(J(Q,W;C))|$.
So $$v_\alpha=q^{\frac{1}{2}\innerprod{\alpha,\alpha}_{J_C}}\frac{|\Rep_\alpha(J(Q,W;C))|}{|\GL_\alpha|}.$$
\end{lemma}


For any stability $\nu$, the trace function $\omega$ restricts to $\Rep_\alpha^\nu(Q)\to K$ and descends to the GIT moduli space $\Mod_\alpha^\nu(Q)$.
Note that for $K=\mb{C}$,
$[\omega^{-1}(0)]-[\omega^{-1}(1)]$ is the vanishing cycle of $\omega$ on $\Mod_\alpha^\nu(Q)$. 
Indeed, since the $\mb{C}^*$-action is induced from a cut, the primitive character condition is trivially satisfied.
To apply \cite[Proposition 1.11]{BBS}, we can verify as in \cite[Lemma 3.2]{N2} that $\lim_{t\to 0} tx$ exists in $\Mod_\alpha^\nu(Q)$ for any $x\in \Mod_\alpha^\nu(Q)$. 

Apply the Hall character $\int_\omega$ to the identity \eqref{eq:HallID}, then we obtain \cite[Theorem 5.7]{M2}:
\begin{proposition} \label{C:int1}
$$\frac{|\varphi_{\omega}(\Rep_\alpha^\nu(Q))|_\vir}{|\GL_\alpha|_\vir}=\sum_* (-1)^{s-1}q^{\frac{1}{2}\sum_{i>j}(\alpha_i,\alpha_j)}\prod_{k=1}^s v_{\alpha_k}(q),$$
where the summation $*$ is the same as in Lemma \ref{L:HNid}.
\end{proposition}

We denote by $v_\alpha^\nu(q)$ the above rational function in $q^{\frac{1}{2}}$.
Following \cite{Fc2}, we say an algebra $A$ is polynomial-count if each $\Rep_\alpha(A)$ is polynomial-count.

\begin{corollary} Assume that the GIT quotient $\Mod_\alpha^\nu(Q)$ is geometric. If $J(Q,W;C)$ is polynomial-count, then so is $\varphi_\omega(\Mod_\alpha^\nu(Q))$.
\end{corollary}

\begin{definition} A pair $(\alpha,\nu)$ is called {\em numb} to a cut $C$ on $Q$ if
the vector bundle $\pi: \Rep_\alpha(Q) \to \Rep_\alpha(Q_C)$ restricts to $\nu$-semistable representations.
\end{definition}

Later we will need the following generalization of Lemma \ref{L:VtoR}. The proof is the same as that in \cite{N2}. For readers' convenience, we copy the proof here.

\begin{lemma} \label{L:VtoRmu} If $(\alpha,\nu)$ is numb to $C$, then $|\varphi_\omega(\Rep_\alpha^\nu(Q))|=q^{\innerprod{\alpha,\alpha}_C}|\Rep_\alpha^\nu(J(Q,W;C))|$.
So $$v_\alpha^\nu=q^{\frac{1}{2}\innerprod{\alpha,\alpha}_{J_C}}\frac{|\Rep_\alpha^\nu(J(Q,W;C))|}{|\GL_\alpha|}.$$
\end{lemma}

\begin{proof} By assumption, $\pi: \Rep_\alpha^\nu(Q) \to \Rep_\alpha^\nu(Q_C)$ is a vector bundle of rank $d=\innerprod{\alpha,\alpha}_C$. The restriction of $\omega$ to the fibre $\pi^{-1}(M)$ is zero if $M\in\Rep_\alpha^\nu(J(Q,W;C))$, and is a non-zero linear function if $x\notin \Rep_\alpha^\nu(J(Q,W;C))$.
Hence $$|\omega^{-1}(0)|=q^d|\Rep_\alpha^\nu(J(Q,W;C))|+q^{d-1}(|\Rep_\alpha^\nu(Q_C)|-|\Rep_\alpha^\nu(J(Q,W;C))|).$$
By \eqref{eq:virtual}, \begin{align*} |\varphi_\omega(\Rep_\alpha^\nu(Q))|&=\frac{q|\omega^{-1}(0)|-|\Rep_\alpha^\nu(Q)|}{q-1},\\
&=q^{\innerprod{\alpha,\alpha}_C}|\Rep_\alpha^\nu(J(Q,W;C))|.
\end{align*}
\end{proof}

Let $\{e_i\}_i$ be the standard basis of $\mb{Z}^{Q_0}$. The {\em $k$-th (absolute) framing stability} $\nu_k$ is the slope function given by $e_k^*/d$, where $d(\alpha)=\sum_{v\in Q_0} \alpha(v)$. It is not hard to see that if all arrows in $C$ end in $k$, then $(\alpha,\nu_k)$ with $\alpha_k=1$ is numb to $C$.


\section{Mutation of Quivers with Potentials} \label{S:QPmu}

The key notion in \cite{DWZ1} is the definition of mutation $\mu_k$ of a quiver with potentials at some vertex $k\in Q_0$. Let us briefly recall it. The first step is to define the following new quiver with potential $\wtd{\mu}_k(Q,W)=(\wtd{Q},\wtd{W})$.
We put $\wtd{Q}_0=Q_0$ and $\wtd{Q}_1$ is the union of three different kinds of arrows
\begin{enumerate}
\item[$\bullet$] all arrows of $Q$ not incident to $k$,
\item[$\bullet$] a composite arrow $[ab]$ from $ta$ to $hb$ for each $a$ and $b$ with $ha=tb=k$, and
\item[$\bullet$] an opposite arrow $a^*$ (resp. $b^*$) for each incoming arrow $a$ (resp. outgoing arrow $b$) at $k$.
\end{enumerate}
The new potential is given by
$$\wtd{W}:=[W]+\sum_{ha=tb=k}b^*a^*[ab],$$
where $[W]$ is obtained by substituting $[ab]$ for each words $ab$ as above occurring in (any cyclically equivalent) $W$.

Let $A$ be the algebra $J(Q,W;C)$ and $T$ be the representation $A/P_k\oplus \tau^{-1} S_k$, where $\tau$ is the classical AR-transformation \cite{ASS}.
Clearly, $\tau^{-1} S_k$ can be presented as $\displaystyle P_k\xrightarrow{(a)_a} \bigoplus_{ha=k} P_{ta}\to \tau^{-1} S_k\to 0$, where $(a)_a$ is the row vector with entries arrows pointing to $k$.
Recall that an $A$-module $T$ is called {\em tilting} if $T$ has finite projective dimension, $\Ext_A^i(T,T)=0$ for all $i>0$, and there is an exact sequence
$$0\to A \to T_1\to T_2\to \cdots \to T_n \to 0,$$
where each $T_i$ is a finite direct sum of direct summands of $T$.
\begin{lemma}[{\cite[Corollary 2.2.b]{L}}] \label{L:BBT} $T$ is a tilting module if and only if the map $\displaystyle P_k\xrightarrow{(a)_a} \bigoplus_{ha=k} P_{ta}$ is injective.
\end{lemma}
\noindent In this case, $T$ is called the {\em BB-tilting module} at $k$. The dual notion of $T$ is the BB-cotilting module $T^\vee=A^*/I_k\oplus \tau S_k$.
What we desire is the following nice situation:\\
\fbox{
\parbox{\textwidth}{$\smiley\quad\quad$There is an algebraic cut $\wtd{C}$ on $(\wtd{Q},\wtd{W})$ such that
$J(Q,W;C)$ and $J(\wtd{Q},\wtd{W};\wtd{C})$ are tilting equivalent via the functor $\Hom_A(T,-)$ or $\Hom_A(-,T^\vee)$.}  }

In general, the existence of another (not necessarily algebraic) cut on $\wtd{W}$ is not guaranteed. However, if we assume that
\begin{equation}\label{A:cut}\text{all arrows ending at $k$ do not belong to a cut $C$,} \end{equation}
then we can assign a new cut $\wtd{C}$ containing all
\begin{enumerate} \item[$\bullet$] $c\in C$ if $tc\neq k$,
\item[$\bullet$] arrows $b^*$ if $b\notin C$,
\item[$\bullet$] composite arrows $[ab]$ with $b\in C$.
\end{enumerate}
This definition is the graded right mutation defined in \cite{AO} adapted to our setting.
There is a graded version of splitting theorem (\cite[Theorem 4.6]{DWZ1}). 
Recall that a (graded) QP $(Q,W)$ is {\em trivial} if the potential $W$ is in the space $KQ_2$ spanned by paths of length $2$, and if the Jacobian algebra $J(Q,W)$ is isomorphic to the semisimple algebra $KQ_0$. A (graded) QP $(Q,W)$ is {\em reduced} if $W \cap KQ_2$ is zero.
Applied to QP with a cut, we have
\begin{lemma} \label{L:Gsplit} $(Q,W,C)$ is graded right-equivalent to the direct sum $(Q_{\op{red}},W_{\op{red}},C_{\op{red}})\oplus (Q_{\op{triv}},W_{\op{triv}},C_{\op{triv}})$,
where $(Q_{\op{red}},W_{\op{red}},C_{\op{red}})$ is {\em reduced} and $(Q_{\op{triv}},W_{\op{triv}},C_{\op{triv}})$ is {\em trivial}, both unique up to graded right-equivalence.
\end{lemma}

We denote the reduced part of $(\wtd{Q},\wtd{W},\wtd{C})$ by $\mu_k(Q,W,C):=(Q',W',C')$.

\begin{theorem} \label{T:tilting} Assume that $C$ is a cut satisfying Definition \ref{D:exact}.(2) and \eqref{A:cut}, and that $\Ext_A^3(S_i,S_k)=0$ for any $i\neq k$. Then $J(Q,W;C)$ is tilting equivalent to $J(\wtd{Q},\wtd{W};\wtd{C})$ via $\Hom_A(T,-)$.
\end{theorem}

\begin{corollary} \label{C:tilting} If $C$ is an algebraic cut satisfying \eqref{A:cut},
then $C'$ is also algebraic, and $J(Q,W;C)$ is tilting equivalent to $J(\wtd{Q},\wtd{W};\wtd{C})$ via $\Hom_A(T,-)$.
\end{corollary}

These slightly generalize the main results of \cite{M}. We will prove them in the appendix. By Lemma \ref{L:Gsplit}, the above $J(\wtd{Q},\wtd{W};\wtd{C})$ can be replaced by $J(Q',W';C')$. If we want to work with the assumption dual to \eqref{A:cut}, that is, all arrows starting with $k$ do not belong to $C$, then we should take the functor $\Hom_A(-,T^\vee)$.

The equivalence $\Hom_A(T,-)$ induces a map $\phi_k$ in the corresponding $K_0$-group
\begin{align} \label{eq:lineariso} \phi_k([S_i])=\begin{cases}
[S_i'] & i\neq k, \\
-[S_k']+\sum_{ha=k}[S_{ta}'] & i=k;\end{cases}
\end{align}
and its dual $\Hom_A(-,T^\vee)$ induces $\phi_k^\vee$ given by
\begin{align*} \phi_k^\vee([S_i])=\begin{cases}
[S_i'] & i\neq k, \\
-[S_k']+\sum_{tb=k}[S_{hb}'] & i=k.\end{cases}
\end{align*}
Since the $K_0$-groups of $\module(J(Q,W;C))$ and $\module(J(Q',W';C'))$ can be identified with $\mb{Z}^{Q_0}$, by slight abuse of notation, we also write $\phi_k$ and $\phi_k^\vee$ for the corresponding linear isometries on $\mb{Z}^{Q_0}$.
Due to the equivalence, we have that $\innerprod{\alpha,\beta}_{J_C}=\innerprod{\phi_k{\alpha},\phi_k{\beta}}_{J_C'}$.
Moreover, it is easy to verify that $(\alpha,\beta)=(\phi_k\alpha,\phi_k\beta)'$, or equivalently, $B'=\phi_k B \phi_k^\T$, where $(-,-)'$ is the antisymmetric form of $Q'$.

We denote \begin{align*}
\module(A)_k:=\{M\in\module A \mid \Hom_{A}(S_k,M)=0\}, \\
\module(A)^k:=\{M\in\module A \mid \Hom_{A}(M,S_k)=0\}.
\end{align*}
Note that under the assumption $\smiley$, $\module(J(Q,W;C))^k$ \big(resp. $\module(J(Q',W';C'))_k$\big) is the torsion (resp. torsion-free) class determined by the tilting module $T$ \cite[VI.2]{ASS}.
So $$\module(J(Q,W;C))^k\cong \module(J(Q',W';C'))_k.$$
In particular, for $\alpha'=\phi_k(\alpha)$ we have that \begin{equation} \label{eq:Repmu}
\frac{|\Rep_\alpha(J(Q,W;C))^k|}{|\GL_\alpha|}=\frac{|\Rep_{\alpha'}(J(Q',W';C'))_k|}{|\GL_{\alpha'}|}.\end{equation}


%

\section{Wall-crossing Formula} \label{S:WC}
Let $\innerprod{S_k}$ be the subcategory of $\module KQ$ generated by the simple $S_k$. 
\begin{definition} We denote $\mb{E}_k:=\int_\omega\chi(\innerprod{S_k})=\sum_n \frac{q^{n^2/2}}{|\GL_n|}x_k^n=\exp_q\Big(\frac{q^{1/2}}{q-1}x_k\Big)$.
\end{definition}

Recall the generating series $\mb{V}(Q,W)$ defined before Lemma \ref{L:VtoR} and the quantum Laurent series algebra $X_{(Q)}$ defined before Lemma \ref{L:int}.
We set $\mb{V}:=\mb{V}(Q,W)\in X_{(Q)}$ and $\mb{V}':=\mb{V}(\mu_k(Q,W))\in X_{(Q')}$.
Let $\Phi_k$ (resp. $\Phi_k^\vee$) be the ring homomorphism $X_{(Q)} \to X_{(Q')}$ defined by $x^\alpha \mapsto (x')^{\phi_k\alpha}$ \big(resp. $(x')^{\phi_k^\vee\alpha}$\big).

\begin{theorem} \label{T:WC} Assuming the condition $\smiley$, we have that $$\mb{E}_k'\Phi_k^\vee(\mb{V}\mb{E}_k^{-1})=\mb{V}'=\Phi_k(\mb{E}_k^{-1}\mb{V})\mb{E}_k'.$$
\end{theorem}

\begin{proof} We apply the character $\int_\omega$ to the torsion-pair identity in $H(Q)$ (Lemma \ref{L:torsionid})
$$\chi(\module(Q)_k)\chi(\innerprod{S_k})=\chi=\chi(\innerprod{S_k})\chi(\module(Q)^k).$$
By Lemma \ref{L:intomega} we get
\begin{align*}
\mb{E}_k^{-1}\mb{V}&=\int_\omega \chi(\module(Q)^k)\\
&=\sum_\alpha q^{\frac{1}{2}\innerprod{\alpha,\alpha}_{J_C}} \frac{|\Rep_\alpha(J(Q,W;C))^k|}{|\GL_\alpha|}x^\alpha, & (\text{similar to Lemma \ref{L:VtoRmu}})\\
&=\sum_\alpha q^{\frac{1}{2}\innerprod{\phi_k(\alpha),\phi_k(\alpha)}_{J_C'}} \frac{|\Rep_{\phi_k(\alpha)}(J(Q',W';C'))_k|}{|\GL_{\phi_k(\alpha)}|}x^{\alpha},& (\ref{eq:Repmu})
\intertext{Similarly}
\mb{V}'\mb{E}_k'^{-1} &=\sum_\alpha q^{\frac{1}{2}\innerprod{\alpha,\alpha}_{J_C'}} \frac{|\Rep_\alpha(J(Q',W';C'))_k|}{|\GL_\alpha|}(x')^\alpha,\\
\intertext{Hence}
\mb{V}'&=\Phi_k(\mb{E}_k^{-1}\mb{V})\mb{E}_k'.
\end{align*}
The other half is similar.
\end{proof}

%
%
%
%

\subsection*{Framing}

We freeze a vertex $\infty$ of $Q$, that is, we do not allow to mutate at $\infty$. Let $\module_{0}(Q)$ be all modules supported outside $\infty$. Note that $\module_{0}(Q)$ is an exact subcategory of $\module(Q)$. In particular, it is a torsion-free class, and let $\T_0(Q)$ be its corresponding torsion class.

Let $\mb{T}:=\int_\omega \chi(\T_0(Q))$ and $\mb{V}_0:=\int_\omega \chi(\module_0(Q))$. It follows from the torsion pair identity that
$$\mb{T}=\mb{V}_{0}^{-1}\mb{V}.$$
We keep the assumption $\smiley$. According to Theorem \ref{T:WC}, we have that \begin{align*}
\mb{E}_k'\Phi_k^\vee(\mb{V}\mb{E}_k^{-1})&=\mb{V}'=\Phi_k(\mb{E}_k^{-1}\mb{V})\mb{E}_k'.
\intertext{From the second equality, we get that}
{\mb{V}_{0}'}^{-1}{\mb{V}}' &= {\mb{E}_k'}^{-1}\Phi_k(\mb{V}_0^{-1}\mb{E}_k)\Phi_k(\mb{E}_k^{-1}\mb{V})\mb{E}_k',\\
&={\mb{E}_k'}^{-1}\Phi_k(\mb{V}_0^{-1}\mb{V})\mb{E}_k',
\intertext{hence we get a formula for $\mb{T}':=\int_{\omega'}\chi(\T_0(Q'))$}
\mb{T}' & = {\mb{E}_k'}^{-1}{\Phi}_k(\mb{T}) \mb{E}_k'.
\intertext{Similarly using the first equality, we obtain that}
\mb{T}' & = {\Phi}_k^\vee\big( \mb{E}_k\mb{T}\mb{E}_k^{-1}\big).
\end{align*}

We can also treat $\module_0(Q)$ as a torsion class, and work with its torsion-free class $\F_0$.
If we set $\mb{F}=\int_\omega \F_0$, then $\mb{F}=\mb{V}\mb{V}_0^{-1}$, and we have the dual formula
$$\mb{E}_k'{\Phi}_k^\vee(\mb{F}) {\mb{E}_k'}^{-1}=
\mb{F}' = {\Phi}_k( \mb{E}_k^{-1}\mb{F}\mb{E}_k).$$

Consider the subcategory $\T_0^1(Q)$ of $\T_0(Q)$, which contains all representations having dimension one at the vertex $\infty$.
It is well-known that the category $\T_0^1(Q)$ contains exactly the $\nu_\infty$-stable representations with dimension $1$ at the vertex $\infty$. Let
\begin{equation}\label{eq:TQW} \T(Q,W):=(q^{\frac{1}{2}}-q^{-\frac{1}{2}})\int_\omega \chi(\T_0^1(Q)).
\end{equation}

Since the dimension vector $(1,\beta)$ is coprime to the slope function $\nu_\infty$, the moduli space $\Mod_{(1,\beta)}^{\nu_\infty}(Q)$ is a geometric quotient, and thus we have
\begin{align} \notag \T(Q,W)&=(q^{\frac{1}{2}}-q^{-\frac{1}{2}})\frac{|\varphi_{\omega}(\Rep_{(1,\beta)}^{\nu_\infty}(Q))|_\vir}{|\GL_{(1,\beta)}|_\vir}x^{(1,\beta)}.\\
\label{eq:T} &=\sum_\beta \left|\varphi_{\omega}(\Mod_{(1,\beta)}^{\nu_\infty}(Q))\right|_\vir\ x^{(1,\beta)}.
\end{align}

Now we replace all $\T_0(Q)$ by $\T_0^1(Q)$ in the above argument for $\mb{T}'$, and we can easily see that
\begin{theorem} \label{T:FWC} Assuming the condition $\smiley$, we have that
\begin{equation} \label{eq:Fmu}
{\mb{E}_k'}^{-1}{\Phi}_k(\T) \mb{E}_k'=
\T' = {\Phi}_k^\vee (\mb{E}_k\T\mb{E}_k^{-1}).
\end{equation}\end{theorem}

In the next section, by abuse of notation, we will write $\mu_k$ for the operator $\Ad^{-1}(\mb{E}_k')\circ\Phi_k$, and $\mu_k^\vee$ for the operator $\Phi_k^\vee\circ\Ad(\mb{E}_k)$.

\section{Application to Cluster Algebras} \label{S:cluster}

Let $B$ be an $n\times m$ matrix with $n\leq m$. The principal part $B_{\op{p}}$ of $B$ is by definition the left $n\times n$ submatrix.  
We assume that $B_{\op{p}}$ is skew-symmetric.
Let $\Lambda$ be another skew-symmetric matrices of size $m\times m$.  
We assume that $\Lambda$ and $B$ are {\em unitally compatible}, that is, ${B} \Lambda = (-I_n,0)$.

We can associate to $B$ an (ice) quiver $Q$ without loops and 2-cycles satisfying \eqref{eq:bij}.
The vertices in $[n+1, m]$ are {\em frozen vertices}. We denote by $Q_{\op{p}}$ {\em the principal part} of $Q$, that is, the full subquiver of $Q$ by forgetting all frozen vertices.
The matrix $B$ is called the {\em $B$-matrix} of $Q$.

Let $X_\Lambda$ be the quantum Laurent polynomial ring $\mb{Z}[q^{\pm\frac{1}{2}}][x_1^{\pm 1},x_2^{\pm 1},\cdots,x_m^{\pm 1}]$ with multiplication given by
\begin{equation} \label{eq:QLPA} x^\alpha x^\beta = q^{\frac{1}{2}\Lambda(\alpha,\beta)} x^{\alpha+\beta}.
\end{equation}
Here, we write $\Lambda(-,-)$ for the associated bilinear form of $\Lambda$.
As an Ore domain \cite[Appendix]{BZ}, $X_\Lambda$ is contained in its skew-field of fractions $\mc{F}(X_\Lambda)$.

\begin{definition} A toric frame is a map $X:\mb{Z}^m\to \mc{F}(X_\Lambda)$, such that $X(\alpha)=\rho(x^{\eta(\alpha)})$ for some automorphism $\rho$ of the skew-field $\mc{F}(X_\Lambda)$, and some automorphism $\eta$ of the lattice $\mb{Z}^m$.
\end{definition}


By abuse of notation we can view $X_\Lambda$ naturally as the toric frame: $X_\Lambda(\alpha)=x^\alpha$. Let $\{e_i\}_{1\leq i\leq m}$ be the standard basis of $\mb{Z}^m$. We also denote by $\phi_k$ the matrix of the linear isometry \eqref{eq:lineariso}, and by $\phi_k^{\op{p}}$ its restriction on the principal part $Q_{\op{p}}$. For any integer $b$, we write $[b]_+$ for $\max(0,b)$.

\begin{definition} A {\em seed} is a triple $(\Lambda,B,X)$ such that $X(g)X(h)=q^{\frac{1}{2}\Lambda(\alpha,\beta)}X(g+h)$ for all $g,h\in\mb{Z}^m$.
The {\em mutation} of $(\Lambda,{B},X)$ at $k$ is a new triple $(\Lambda',{B}',X')=\mu_k(\Lambda,{B},X)$ defined by
\begin{align}
\label{eq:mu} (\Lambda',{B}')&=({\phi}_k^\T \Lambda {\phi}_k, \phi_k^{\op{p}} {B} {\phi}_k^\T),
\intertext{and $X'$ is detetmined by the following {\em exchange relation}}
\label{eq:mu1} X'(e_k) &= X\big(\sum_{1\leq j \leq m}[-b_{kj}]_+ e_j - e_k \big) + X\big(\sum_{1\leq j \leq m}[b_{kj}]_+ e_j - e_k \big),  \\
\label{eq:mu2} X'(e_j) &= X(e_j) \qquad 1\leq j \leq m, j\neq k.
\end{align}
\end{definition}

Since $\phi_k=\phi_k^{-1}$, we see that $(\Lambda',B')$ is also unitally compatible. The automorphism $\rho$ for $X'$ was constructed explicitly in \cite[Proposition 4.2]{BZ}. One should notice that the mutation $\mu_k$ is an involution.

Let $\mb{T}_n$ be the {\em $n$-regular tree} with root $t_0$. There is a unique way of associating a seed $(\Lambda_t,{B}_t,X_t)$ for each vertex $t\in\mb{T}_n$ such that
\begin{enumerate} \item $(\Lambda_{t_0},{B}_{t_0},X_{t_0})=(\Lambda,{B},X_\Lambda)$,
\item if $t$ and $t'$ are linked by an edge $k$, then the seed $(\Lambda_{t'},{B}_{t'},X_{t'})$ is obtained from $(\Lambda_{t},{B}_{t},X_{t})$ by the mutation at $k$.
\end{enumerate}

\begin{definition} The quantum cluster algebra $\mc{C}(\Lambda,{B})$ with initial seeds $(\Lambda,{B},X_\Lambda)$ is the $\mb{Z}[q^{\pm 1}]$-subalgebra of $\mc{F}(X_\Lambda)$ generated by all {\em cluster variables} $X_t(e_i)\ (1\leq i\leq n)$, {\em coefficients} $X_t(e_i)$ and $X_t(e_i)^{-1}\ (n+1\leq i\leq m)$.
\end{definition}

Recall the operators $\mu_k=\Ad^{-1}(\mb{E}_k')\circ\Phi_k$ and $\mu_k^\vee=\Phi_k^\vee\circ\Ad(\mb{E}_k)$.
Moreover $\mb{E}(y)=\exp_q\Big(\frac{q^{1/2}}{q-1}y\Big)$ can also be written as the formal product 
\footnote{The inverse of this product is the $q$-Pochhammer symbol $(-q^{\frac{1}{2}}y;q)_\infty$.}
$$\prod_{l=0}^\infty (1+q^{l+\frac{1}{2}}y)^{-1}.$$
It satisfies
$$\mb{E}(q^{\pm 1}y)=(1+q^{\pm \frac{1}{2}} y)^{\pm 1}\mb{E}(y).$$

Let $Y_{(Q_{\op{p}})}$ be the quantum Laurent polynomial (rather than Laurent series) algebra in variables $\{y_i\}_{i\in Q_0}$ having the same multiplication rule as $X_{(Q_{\op{p}})}$, that is, $y^\beta y^\gamma = q^{-\frac{1}{2}(\beta,\gamma)} y^{\beta+\gamma}$. 
From the fact that 
$$y_i\mb{E}(y_k) = \mb{E}(q^{b_{ki}} y_k)y_i\ \text{ and }\ \mb{E}(y_k)^{-1}\mb{E}(q^by_k)=\prod_{l=1}^b (1+q^{l-\frac{1}{2}}y_k),$$
we can easily deduce the following $Y$-seeds mutation formula

\begin{lemma}[{\cite[(4.11)]{Ke2}}] \label{L:muyi}  $$\mu_k^\vee(y_i)=\mu_k(y_i)=\begin{cases}y_k^{-1} & (i=k), \\
\displaystyle y^{e_i+[-b_{ik}]_+e_k}\prod_{l=1}^{|b_{ik}|}(1+q^{\sgn(b_{ik})(l-\frac{1}{2})}y_k)^{\sgn(b_{ik})} & \text{otherwise}.
\end{cases}$$
\end{lemma}

Let $\mc{F}(Y_{(Q_{\op{p}})})$ be its skew-field of fraction of $Y_{(Q_{\op{p}})}$.
From the above formula, we see that applying a sequence of mutation operators to a Laurent polynomial, we get an element in $\mc{F}(Y_{(Q_{\op{p}})})$ rather than an arbitrary series.
We consider the lattice map $\mb{Z}^n\to \mb{Z}^m, \beta \mapsto \beta B$. This map induces an operator $$\op{b}:Y_{(Q_{\op{p}})} \to X_\Lambda,\ y^\beta\mapsto x^{\beta B}.$$ 
By the unital compatibility of $\Lambda$ and ${B}$, we have that $\alpha B \beta^\T=\Lambda(\alpha{B},\beta {B})$. So we conclude that
\begin{lemma} \label{L:bhomo} The operator $\op{b}$ is an algebra homomorphism, and thus induces a skew-field homomorphism $\op{b}:\mc{F}(Y_{(Q_{\op{p}})}) \to \mc{F}(X_\Lambda)$.
\end{lemma}

Let $\bold{k}_s:=(k_1,k_2,\dots,k_s)$ be a sequence of edges connecting $t_0$ and $t_s$.
We write $\mu_{\bold{k}_s}$ for the sequence of mutation $\mu_{k_s}\cdots\mu_{k_2}\mu_{k_1}$.
For simplicity, we write $B_{r}$ for $B_{t_r}$ and $X_{r}$ for $X_{t_r}$.
The next lemma says that the operator $\op{b}$ is compatible with mutations.
\begin{lemma} \label{L:muxy} $\op{b}\circ\mu_{\bold{k}_s}^{-1}(y^{\beta})=X_{s}(\beta{B}_s)$ for any $\beta\in\mb{Z}^n$.
\end{lemma}

\begin{proof} Using the unital compatibility of $\Lambda$ and ${B}$, this is clearly reduced to prove for $\beta=e_i$.
We prove by induction on $s$. For $s=0$, it is trivial. Suppose that it is true for $s$, then
\begin{align*} \op{b}\circ\mu_{\bold{k}_{s+1}}^{-1}(y^{e_i})
&= \op{b}\circ\mu_{\bold{k_{s}}}^{-1}(\mu_{k_{s+1}}^{-1}y^{e_i}),\\
&\stackrel{\text{Lemma \ref{L:muyi}}}{\hbox{\equalsfill}} \op{b}\circ\mu_{\bold{k_{s}}}^{-1} \Big(y^{e_i+[b_{ki}^{s+1}]_+e_k} \prod_{l=1}^{|b_{ik}^{s+1}|} \big(1+q^{\sgn(b_{ik}^{s+1})(l-\frac{1}{2})}y^{e_k}\big)^{\sgn(b_{ik}^{s+1})}\Big),\\
&\stackrel{\text{Lemma \ref{L:muxy}}}{\hbox{\equalsfill}} X_{s}(e_iB_{s}+[b_{ik}^{s}]_+e_kB_{s}) \prod_{l=1}^{|b_{ki}^{s}|} \big(1+q^{\sgn(b_{ki}^{s})(l-\frac{1}{2})}X_{s}(e_kB_{s})\big)^{\sgn(b_{ki}^{s})}.
\intertext{On the other hand,}
X_{s+1}(e_iB_{s+1}) &= X_{s+1}\big(\sum_j b_{ij}^{s+1} e_j \big),\\
&= q^{\frac{1}{2}\Lambda_{s+1}(e_iB_{s+1}-b_{ik}^s e_k,b_{ik}^s e_k)}X_{s+1}\big(\sum_{j\neq k} b_{ij}^{s+1} e_j \big) X_{s+1}(b_{ik}^{s+1} e_k),\\
&= X_{s}\big(\sum_{j\neq k} b_{ij}^{s+1} e_j \big) \cdot \Big(X_{s}\big(\sum_{1\leq j \leq m} [b_{-kj}^{s}]_+e_j-e_k \big) + X_{s}\big(\sum_{1\leq j \leq m} [b_{kj}^{s}]_+ e_j-e_k\big)\Big)^{b_{ki}^{s}},\\
&= X_{s}\big(\sum_{j\neq k} b_{ij}^{s+1} e_j \big) \cdot \Big(X_{s}\big(\sum_{1\leq j \leq m} [b_{-kj}^{s}]_+e_j-e_k \big) \big(1 + q^{\frac{1}{2}}X_{s}(e_kB_{s})\big)\Big)^{b_{ki}^{s}},\\
&= X_{s}\big(\sum_{j\neq k} (b_{ij}^{s}+[b_{ik}^{s}]_+b_{kj}^{s} - b_{ki}^{s}[-b_{kj}^{s}]_+)e_j\big) \cdot X_{s}\big(b_{ki}^{s}\big(\sum_{1\leq j \leq m} [-b_{kj}^{s}]_+e_j-e_k\big)\big)\cdot \\
& \qquad \prod_{l=1}^{|b_{ki}^{s}|} \big(1+q^{\sgn(b_{ki}^{s})(l-\frac{1}{2})}X_{s}(e_k B_s)\big)^{\sgn(b_{ki}^{s})},\\
&= X_{s}(e_iB_{s}+[b_{ik}^{s}]_+e_kB_{s}) \prod_{l=1}^{|b_{ki}^{s}|} \big(1+q^{\sgn(b_{ki}^{s})(l-\frac{1}{2})}X_{s}(e_kB_{s})\big)^{\sgn(b_{ki}^s)}.
\end{align*}
\end{proof}

For any $t\in \mb{T}_n$, there is a unique sequence of edges $\bold{k}_t$ connecting $t_0$ and $t$.
Let $W$ be some {\em non-degenerate} (\cite[Definition 7.2, Proposition 7.3]{DWZ1}) potential of $Q$, and set $(Q_t,W_t)=\mu_{\bold{k}_t}({Q},W)$.
We shall assume the following condition for $W$:
\begin{equation} \label{Ass_potential}
\text{For any $t\in \mb{T}_n$, and any $k\in Q_t$, the assumption $\smiley$ holds for $({Q}_t,W_t)$.}
\end{equation}
We do not know if such a potential exists for any quiver without loops or $2$-cycles.

To give another definition of $X_t(g)$ for $g\in\mb{Z}_{\geq 0}^m$, we consider the extended QP $(Q_t^g,W_t)$ from $(Q_t,W_t)$ by adding a new vertex $\infty$ and $g_i$ new arrows from $i$ to $\infty$. 
We apply the inverse of $\mu_{\bold{k}_t}$ to $(Q_t^g,W_t)$, and obtain a QP $(Q^g,W^g):=\mu_{\bold{k}_t}^{-1}(Q_t^g,W_t)$. Let $\omega^g$ be the trace function corresponding to the potential $W^g$.

%
%

We freeze the same set of vertices of $Q^g$ as that of $Q$. Although the extended vertex $\infty$ is not frozen, we will never perform mutation at $\infty$ henceforth.
Let $\wtd{B}$ be the $B$-matrix of $Q^g$. Note that $Q$ is the full subquiver of $Q^g$ without the vertex $\infty$ so $\wtd{B}$ is obtained from $B$ by adjointing a row corresponding to the vertex $\infty$.

\begin{definition} \label{D:cluster} For any $g\in\mb{Z}_{\geq 0}^m$, we define $X_t(g)=\op{b}\big( \T(\mu_{\bold{k}_t}^{-1}(Q_t^{g},W_t)) \big),$ where $\T(Q,W)$ is defined in \eqref{eq:TQW}. 
\end{definition}

By Theorem \ref{T:FWC}, we have that
\begin{equation}\label{eq:mut} \T(\mu_{\bold{k}_t}^{-1}(Q_t^{g},W_t))=\mu_{\bold{k}_t}^{-1} \Big((q^{1/2}-q^{-1/2})\frac{y_\infty}{q^{1/2}-q^{-1/2}}\Big)
=\mu_{\bold{k}_t}^{-1}(y_\infty).
\end{equation}


\begin{theorem} \label{T:cluster}
Under the assumption \eqref{Ass_potential},
definition \ref{D:cluster} defines the quantum cluster algebra $\mc{C}(\Lambda,{B})$.
In particular, we have the following {\em quantum cluster character}
$$X_t(g)=\sum_\beta \left|\varphi_{\omega^g}(\Mod_{(1,\beta)}^{\nu_\infty}(Q^g))\right|_\vir\ x^{(1,\beta)\wtd{B}}.$$
\end{theorem}

\begin{proof} 
We trivially extend $\Lambda$ to $\wtd{\Lambda}:=\sm{0& 0\\0& \Lambda}$, and thus we have the natural embedding $X_{\Lambda}\hookrightarrow X_{\wtd{\Lambda}}$. Note that $\wtd{B}\wtd{\Lambda}=\sm{0 & 0 & 0\\ 0 & -I & 0}$, so they are {\em not} unitally compatible.
We say that they are unitally compatible on the principal part.
However, it still makes perfect sense if we define $(\wtd{B}',\wtd{\Lambda}')$ by \eqref{eq:mu} and $\wtd{X}'(e_i)$ by the relation \eqref{eq:mu1}--\eqref{eq:mu2}.
Clearly, for any $k\neq \infty$,
$(\wtd{B}',\wtd{\Lambda}')$ is also unitally compatible on the principal part. This is all we need for the following analogue \eqref{eq:muxy} of Lemma \ref{L:muxy} to hold:
For each $t\in \mb{T}_n$, we associate as before $\wtd{X}_t$, then
\begin{equation}\label{eq:muxy}
\op{b}\circ\mu_{\bold{k}_t}^{-1}(y^{\beta})=\wtd{X}_{t}(\beta \wtd{B}_t).\end{equation}
Moreover, $\wtd{\Lambda}'$ extends $\Lambda'$ in the same way: $\wtd{\Lambda}':=\sm{0& 0\\0& \Lambda'}$ so that we have the natural embedding $X_t(\mb{Z}^{m})\hookrightarrow\wtd{X}_t(\mb{Z}^{m+1})$ for each $t\in\mb{T}_n$.
Hence, $$\op{b}\circ\mu_{\bold{k}_t}^{-1}(y_\infty)=\wtd{X}_{t}(e_\infty\wtd{B}_t)=X_t(g).$$
The last statement on the explicit formula of $X_t(g)$ follows from \eqref{eq:mut}, \eqref{eq:T}, and Lemma \ref{L:bhomo}.
\end{proof}

\begin{remark} We can view $(1,\beta)\wtd{B}$ as $\beta B - g_t$, where $g_t$ is the {\em extended $g$-vector} corresponding to the mutated cluster monomial.
\end{remark}

\begin{example} \label{ex:A3} Consider the quiver
$$\ctrianglcoext{a}{b}{c}$$ with potential $abc$.
We perform a sequence of mutations $\{1,2,3,1\}$, and obtain the quiver
$$\ctrianglext{a}{b}{c}$$ with the same potential. We choose $c$ as the cut.
It is easy to count the vanishing cycles for each dimension vector. For example, for $\beta=(1,1,1)$,
$$|\varphi_{\omega^g}(\Mod_{(1,\beta)}^{\nu_\infty}(Q^g))|=q(2q+2).$$
Note that $2q+2$ counts neither the representation Grassmannian of $P_1\oplus P_2\oplus P_3$ of the Jacobian algebra\\
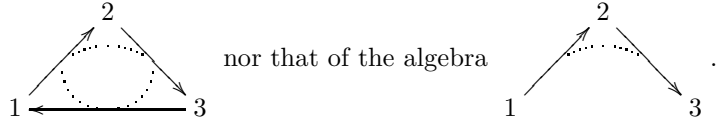
\begin{center}
$\ctriangle{1}{2}{3}{}{}{}$ nor that of the algebra
$\ctrianglecut{1}{2}{3}{}{}{}.$\end{center}
Here, the dotted line between two arrows means a relation given by the vanishing of composition.
In particular, the condition `numb' cannot be removed from Lemma \ref{L:VtoRmu}.
\end{example}


%

\section{Application: Representation Grassmannians and Reflections} \label{S:QGrass}
Let $s$ be a sink of $Q$, and $M$ be a representation of $Q$. We assume that $M$ does not contain the simple representation $S_s$ as a direct summand. Let
$$\T(M):=\sum q^{-\frac{1}{2}\innerprod{\br{M}-\beta,\beta}_Q} {|\Gr^\beta(M)|} x^{(1,\beta)}.$$
We want to compare $\T(M)$ with $\T(\mu_s(M))$.

We say an algebra $A$ is {\em extended from $Q$ by $M$} if $A=KQ[M]:=\sm{KQ & 0\\ M & K}$.
This is an algebra of global dimension two, so we can complete it to a QP $(Q_A,W_A)$ with a cut $C$ such that $J(Q_A,W_A;C)=A$ (see Section \ref{S:QPC}). We freeze the extended vertex $\infty$ of $A$, then

\begin{lemma} $\T(Q_A,W_A)=\T(M)$.
\end{lemma}

\begin{proof} Since all arrows in $C$ end in $\infty$, by Lemma \ref{L:VtoRmu},
$$\T(Q_A,W_A)=\int_\omega \chi(\T_0^1(Q_A))=(q^{\frac{1}{2}}-q^{-\frac{1}{2}})\sum_\alpha q^{\frac{1}{2}\innerprod{(1,\beta),(1,\beta)}_{Q[M]}}\frac{|\Rep_{(1,\beta)}^{\nu_\infty}(Q[M])|}{|\GL_{(1,\beta)}|}x^{(1,\beta)},$$
where $\innerprod{-,-}_{Q[M]}$ is the Euler form of $KQ[M]$, and $\nu_\infty$ is the framing stability.
$\Rep_{(1,\beta)}^{\nu_\infty}(Q[M])$ can be identified with 
$$\{(N,f)\in \Rep_\beta(Q)\times\Hom(M,K^\beta)\mid f\in\Hom_Q(M,N) \text{ is surjective}\}.$$
So the quotient $\Rep_{(1,\beta)}^{\nu_\infty}(Q[M])/\GL_\beta$ is the representation Grassmannian $\Gr^\beta(M)$, and thus 
\begin{align*}\T(Q_A,W_A) &=(q^{\frac{1}{2}}-q^{-\frac{1}{2}})\sum_\beta q^{\frac{1}{2}}q^{-\frac{1}{2}\innerprod{\br{M}-\beta,\beta}_Q}\frac{|\Rep_\alpha^{\nu_\infty}(Q[M])|}{|\GL_\beta||\GL_1|}x^{(1,\beta)},\\
&=\sum_{\beta} q^{-\frac{1}{2}\innerprod{\br{M}-\beta,\beta}_Q}|\Gr^\beta(M)| x^{(1,\beta)}.
\end{align*}
\end{proof}

\begin{remark} The analogous statement does not hold for quivers with potentials in general (see Example \ref{ex:A3}).
\end{remark}

For any $s\in Q_0$ (not necessarily a sink), the cut $C$ of $Q_A$ satisfies the conditions in Corollary \ref{C:tilting}, so we get another algebra $A'=\End_A(T)=J(Q_A',W_A';C')$.

\begin{lemma} The algebra $A'$ is extended from $Q':=\mu_s(Q)$ by $M':=\mu_s(M)$.
\end{lemma}

\begin{proof} Let $\infty$ be the extended vertex of $Q[M]$, and $P_{\hat{\infty}}:=A/P_\infty$. Then $$\End_A(P_{\hat{\infty}})=KQ \text{ and } \Hom_A(P_{\hat{\infty}},P_\infty)=M.$$
Since there are no incoming arrows to $\infty$ for $Q[M]$, we have that $P_{\hat{\infty}} \cong KQ$.
Let $T=A/P_s\oplus T_s$ be the BB-tilting module of $A$ at $s$, and $T_{\hat{\infty}}:=T/P_\infty$.
We need to show that
$$\End_A(T_{\hat{\infty}})=K\mu_s(Q) \text{ and } \Hom_A(T_{\hat{\infty}},P_\infty)=\mu_s(M).$$
Let $T'=KQ/P_s\oplus T_s'$ be the BB-tilting module of $KQ$ at $s$. Since $M$ does not contain $S_s$ as a direct summand
the quiver of $Q[M]$ has no arrow from $\infty$ to $s$. So $T_s=T_s'$, and thus $T_{\hat{\infty}}=T'$. Hence, $\End_A(T_{\hat{\infty}})=K\mu_s(Q)$.
For the second one, we consider \begin{align*} \mu_s(M) &= \Hom_Q(T',\Hom_A(P_{\hat{\infty}},P_\infty)),\\
&=\Hom_A(T'\otimes_{KQ} P_{\hat{\infty}},P_\infty),\\
&=\Hom_A(T_{\hat{\infty}},P_\infty).
\end{align*}
\end{proof}

It follows from the previous two lemmas that
\begin{theorem} \label{T:QGrass} $\T(M)$ and $\T(M')$ are related via \eqref{eq:Fmu}.
In particular, if $M$ is {\em polynomial-count}, that is, all its Grassmannians $\Gr^\beta(M)$ are polynomial-count, then so are all reflection equivalent classes of $M$.
\end{theorem}

\begin{example} \label{ex:333}
Consider the following quiver
$$\vcenter
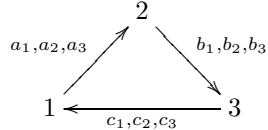
{\xymatrix{
	&  2 \ar[dr]^{b_1,b_2,b_3}   	\\
1 \ar[ur]^{a_1,a_2,a_3}  & &  3 \ar[ll]^{c_1,c_2,c_3} 	 }}$$
with potential $W=\sum_{I:=(i,j,k)\in S_3} (-1)^{\sgn I}a_ib_jc_k$ and an algebraic cut $C=\{c_1,c_2,c_3\}$.
Then the algebra $J(Q,W;C)$ is Beilinson's quiver algebra for $\mb{P}^2$.
This algebra is extended from the quiver $2 \xrightarrow{b_1,b_2,b_3} 3$ by a representation $M$ of dimension $(3,6)$ (see \cite[Example 7.8]{Fc2}).
To compute $\T(M)$, we can first compute $\T(\mu_3(M))$,
where $\mu_3(M)$ can be presented by the following base diagram
$$\Kthreethreethree{b_1}{b_3}{b_2}{b_1}{b_2}{b_3}$$
The black (resp. white) dots are a basis at vertex 3 (resp. vertex 2); The letter on an arrow represents the identity map on the arrow of the same letter.
$$\T(\mu_3(M))=1+x^{(1,0,3)}+x^{(1,3,3)}+[3](x^{(1,0,1)}+x^{(1,0,2)}+x^{(1,1,3)}+x^{(1,2,3)})+3q^{\frac{1}{2}}x^{(1,1,2)}.$$
Here $[n]$ is the quantum number $q^{\frac{1-n}{2}}\big(\frac{q^{n}-1}{q-1}\big)$.
Using Theorem \ref{T:QGrass}, we find that
\begin{align*}\T(M)&=1+x^{(1,3,0)}+x^{(1,3,6)}+[3](x^{(1,1,0)}+x^{(1,2,0)}+x^{(1,2,3)})+3q^{\frac{1}{2}}x^{(1,1,1)}\\
&+[3][3](x^{(1,2,1)}+x^{(1,2,2)})+(q^{\frac{3}{2}}+q^{-\frac{3}{2}})[3](x^{(1,3,1)}+x^{(1,3,5)})\\
&+(q-1+q^{-1})[3][5](x^{(1,3,2)}+x^{(1,3,4)})+(q-1+q^{-1})[4][5]x^{(1,3,3)}.\end{align*}
Employing the methods developed in \cite{Fc2}, we can compute $|\varphi_\omega(\Mod_\alpha^\nu(Q))|$ and $|\Mod_\alpha^\nu(J(Q,W;C))|$ for all $\alpha$ with $\alpha_1=1$ and generic $\nu$.
\end{example}

\section{Appendix: Proof of Theorem \ref{T:tilting}} \label{S:app}
Theorem \ref{T:tilting} generalizes the main result of \cite{M} from APR-tilting modules to BB-tilting modules (see after Lemma \ref{L:BBT}). We slightly simplify its proof as well.
We follow the matrix notation ${_r(-)_c}$ in \cite{M}, that is, we write the row index $r$ and column index $c$ as left and right subscripts respectively.  

\begin{lemma}[{\cite[Proposition 3.3]{BIRS}}] \label{L:presentation} Let $Q$ be a finite quiver and $A$ be a finite dimensional basic algebra. Let $R$ be a set of relations in $Q$, and we assume that any $r\in R$ is a formal linear sum of paths in $Q$ with a common start $tr$ and a common end $hr$.
Then $A$ can be presented as $\widehat{KQ}/\innerprod{R}$ if and only if
there is an algebra homomorphism $\pi:\widehat{KQ}\to A$ such that the sequence
$$\bigoplus_{tr=i} \pi(e_{hr})A \xrightarrow{{_r\big(\pi(a^{-1}r)\big)_a}} \bigoplus_{ta=i} \pi(e_{ha})A \xrightarrow{_a\big(\pi(a)\big)}  \rad(\pi(e_i)A) \to 0$$
is exact for any $i\in Q_0$.
Here, $a^{-1}$ is the formal inverse of $a$ defined by 
$$a^{-1}(a_1a_2\cdots a_m) = \begin{cases} a_2\cdots a_m & \text{if } a_1=a,\\ 0 & \text{otherwise}.\end{cases}$$
\end{lemma}


We set $\displaystyle P_\inn=\bigoplus_{ha=k} P_{ta}$ and $\displaystyle P_\out=\bigoplus_{tb=k,b\notin C} P_{hb}$. Recall from Lemma \ref{L:BBT} that the summand $T_k:=\tau^{-1} S_k$ in the BB-tilting module can be presented as
\begin{equation} \label{eq:Tk} 0\to P_k\xrightarrow{\alpha} P_\inn\xrightarrow{g} T_k\to 0,
\end{equation}
where $\alpha:=(a)_a$ and $g:={}_a(g_a)$.

Using the presentation $\widehat{KQ_C}/\innerprod{\partial_C W}$ of $J(Q,W;C)$.
We have for $i\neq k$
\begin{equation} \cdots\to P' \to \bigoplus_{hc=i,c\in C} P_{tc} \xrightarrow{\partial_{cb}} \bigoplus_{tb=i} P_{hb} \xrightarrow{\beta} P_i \to S_i\to 0,\end{equation}
where $\beta:={}_b(b)$ and $\partial_{cb}:={}_c(\partial_{c}\partial_b W)_b$.
Since the cut $C$ satisfies Definition \ref{D:exact}.(2), the first three terms are part of the minimal projective resolution of $S_i$. We assume that the projective $P'$ is minimal as well.

This fits in the following commutative diagram
\begin{equation*}\xymatrix@C=9ex@R=5ex{
0 \ar[r]	& \displaystyle c_{ki} P_k \ar[r]^{\oplus \alpha} \ar@{^{(}->}[d]^{\iota} & \displaystyle c_{ki} P_{in} \ar[r]^{\oplus g} \ar[d]^{\partial_{acb}} & \displaystyle c_{ki} T_k \ar[r] \ar[d]^{f:=_c(f_c)} & 0	 \\
P'\ar[r]^{h} & \displaystyle \bigoplus_{hc=i,c\in C} P_{tc} \ar[r]^{\partial_{cb}}  & \displaystyle \bigoplus_{tb=i} P_{hb} \ar[r]^{\beta} & P_i \ar[r] & S_i}
\end{equation*}
Here $c_{ki}=|C\cap Q(k,i)|$, and the first row is a direct sum of $c_{ki}$ copies of \eqref{eq:Tk};
The map $\iota$ is the natural embedding, the map $\partial_{acb}$ is given by the matrix $_{a,c}(\partial_{a}\partial_{c}\partial_{b}W)_b$, and $f$ is induced from $\partial_{acb}$. We then take the mapping cone of the above diagram, and cancel out the last term $c_{ki} P_k$. We end up with
\begin{equation}\label{eq:cone} P'\xrightarrow{h} c_{ki} P_\inn \oplus \bigoplus_{hc=i,tc\neq k}P_{tc} \xrightarrow{g':=\sm{\oplus g & \partial_{acb} \\ 0 & \partial_{cb}}} c_{ki} T_k\oplus \bigoplus_{tb=i} P_{hb} \xrightarrow{f':=\sm{-f\\ \beta}} P_i \to S_i\end{equation}

Recall our setting in Section \ref{S:QPmu}. Let $\wtd{Q}_{\wtd{C}}$ be the quiver obtained from $\wtd{Q}$ by forgetting all arrows in $\wtd{C}$.
To apply Lemma \ref{L:presentation}, we construct an algebra homomorphism $\pi:\widehat{K\wtd{Q}_{\wtd{C}}} \to \End_A(T)$ as follows. For any direct summands $T_i,T_j$ of $T$, we will view $\Hom_A(T_i,T_j)$ under the natural embedding into $\Hom_A(T,T)$. Let $\id_i$ be the identity map in $\Hom_A(T_i,T_i)$. We define
\begin{enumerate} \item $\pi(e_i)= \id_i$,
\item $\pi(a)=a\in\Hom_A(P_i,P_j)$ for $i,j\neq k$,
\item $\pi(a^*)=g_a\in \Hom_A(P_{ta},T_k)$,\\
$\pi(b^*)=-f_c\in \Hom_A(T_k,P_{hc})$ for $b\in C$,\\
$\pi([ab])= ba\in\Hom_A(P_{ha},P_{tb})$ for $b\notin C$.
\end{enumerate}

Recall that $\displaystyle \wtd{W}:=[W]+\sum_{ha=tb=k}b^*a^*[ab],$ and
$\wtd{C}$ contains all
\begin{enumerate} \item[$\bullet$] $c\in C$ if $tc\neq k$,
\item[$\bullet$] arrows $b^*$ if $b\notin C$,
\item[$\bullet$] composite arrows $[ab]$ with $b\in C$.
\end{enumerate}
So the corresponding relations $\partial_{\wtd{C}} \wtd{W}$ are given by
\begin{enumerate}
\item[$\bullet$] $R_0=\{\partial_c [W]\}_{tc\neq k}$,
\item[$\bullet$] $R_1=\{a^*[ab]\}_{b\notin C}$,
\item[$\bullet$] $R_2=\{\partial_{ab}W + b^*a^*\}_{b\in C}$.
\end{enumerate}
We will see that Theorem \ref{T:tilting} is an immediate consequence of the following two lemmas.

\begin{lemma} \label{L:res1} We have the following exact sequence
$$\Hom_A(T,P_\out)\xrightarrow{\circ r_1} \Hom_A(T,P_\inn) \xrightarrow{\circ g} \rad(\Hom_A(T,T_k)) \to 0,$$
where $r_1$ is the matrix $_b\{ba\}_a$.
\end{lemma}

\begin{proof} We apply $\Hom_A(T,-)$ to the exact sequence \eqref{eq:Tk}, and get
$$0\to \Hom_A(T,P_k) \xrightarrow{\circ\alpha} \Hom_A(T,P_\inn)\xrightarrow{\circ g} \Hom_A(T,T_k) \to \Ext_A^1(T,P_k)\to \Ext_A^1(T,P_\inn).$$
The last term $\Ext_A^1(T,P_\inn)$ vanishes because the first map below is surjective
$$\Hom_A(P_\inn,P_\inn)\xrightarrow{} \Hom_A(P_k,P_\inn)\xrightarrow{} \Ext_A^1(T_k,P_\inn)\to 0.$$
Next, $\Ext_A^1(T,P_k)$ is one-dimensional because of the following exact sequence
$$\Hom_A(P_\inn,P_k)\xrightarrow{} \Hom_A(P_k,P_k)\to \Ext_A^1(T_k,P_k)\to 0.$$
Finally, we claim the image of $\circ r_1$ is exactly the image of $\circ \alpha$.
By the definition of $r_1$, it suffices to show that $\Hom_A(T,P_\out)\xrightarrow{\circ \beta}\Hom_A(T,P_k)$ is surjective.
But the cokernel of $\circ \beta$ is $\Hom_A(T,S_k)=0$.
\end{proof}

Applying $\Hom_A(T,-)$ to \eqref{eq:cone}, we get the complex
$$\Hom_A(T,c_{ki} P_\inn \oplus \bigoplus_{hc=i,tc\neq k}P_{tc}) \xrightarrow{\circ g'} \Hom_A(T,c_{ki} T_k\oplus \bigoplus_{tb=i} P_{hb}) \xrightarrow{\circ f'} \Hom_A(T,P_i) \to \Hom_A(T,S_i)$$

\begin{lemma} \label{L:res2} If $\Ext_A^3(S_i,S_k)=0$ for any $i\neq k$, this complex is exact and induces
$$\Hom_A(T,c_{ki} P_\inn \oplus \bigoplus_{hc=i,tc\neq k}P_{tc}) \xrightarrow{} \Hom_A(T,c_{ki} T_k\oplus \bigoplus_{tb=i} P_{hb}) \xrightarrow{} \rad(\Hom_A(T,P_i)).$$
\end{lemma}

\begin{proof} We first show that the complex is exact at $\Hom_A(T,c_{ki} T_k\oplus \bigoplus_{tb=i} P_{hb})$. We apply $\Hom_A(T,-)$ to the exact sequence
$$0\to \Img h \to c_{ki} P_\inn \oplus \bigoplus_{hc=i}P_{tc} \to \Img g' \to 0,$$
and get
$$\Hom_A(T,c_{ki} P_\inn \oplus \bigoplus_{hc=i,tc\neq k}P_{tc})\to \Hom_A(T,\Img g')\to \Ext_A^1(T,\Img h).$$
If $\varphi\in\Hom_A(T,c_{ki} T_k\oplus \bigoplus_{tb=i} P_{hb})$ such that $\varphi f'=0$, then $\varphi(T)\subseteq \Img g'$.
So it suffices to show that $\Ext_A^1(T,\Img h)=0$.
The condition $\Ext_A^3(S_i,S_k)=0$ implies that $P'$ has no $P_k$ as its summands. So $\Ext_A^1(T,P')=0$, and thus $$0\to \Ext_A^1(T,\Img h)\to \Ext_A^2(T,\Ker h)=0.$$
Since $\displaystyle c_{ki} P_\inn \oplus \bigoplus_{hc=i,tc\neq k}P_{tc}$ has no $P_k$ as its direct summands, for the same reason the complex is exact at $\Hom_A(T,P_i)$.

We remain to show that the cokernel of $\circ f'$ is one-dimensional.
Let $\Omega S_i$ be the first syzygy of $S_i$. We apply $\Hom_A(T,-)$ to
$$0\to \Omega S_i\xrightarrow{f''} P_i\to S_i\to 0,$$
and obtain
$$\Hom_A(T,\Omega S_i) \xrightarrow{\circ f''} \Hom_A(T,P_i) \to \Hom_A(T,S_i)\to \Ext_A^1(T,\Omega S_i)\to \Ext_A^1(T,P_i)=0.$$
Since $\displaystyle \Ext_A^1(T,c_{ki} P_\inn \oplus \bigoplus_{hc=i,tc\neq k}P_{tc})$ vanishes, the cokernel of $\circ f'$ is the same as that of $\circ f''$.
By applying $\Hom_A(S_i,-)$ to \eqref{eq:Tk}, we see that \begin{equation}
\label{eq:hom} \Hom_A(T,S_i)=\Hom_A(T_k,S_i)\oplus K\cong\Ext_A^1(S_i,S_k)^* \oplus K.
\end{equation}
In the meanwhile,
\begin{align*}
\Ext_A^1(T,\Omega S_i)=\Ext_A^1(\tau^{-1}S_k,\Omega S_i)=\br{\Hom}_A(\Omega S_i,S_k)^*=\Hom_A(\Omega S_i,S_k)^*,\\
 0=\Hom_A(P_i,S_k)\to \Hom_A(\Omega S_i,S_k)\to\Ext_A^1(S_i,S_k)\to \Ext_A^1(P_i,S_k)=0.
 \end{align*}
So
$$\Ext_A^1(T,\Omega S_i)=\Ext_A^1(S_i,S_k)^*.$$
Together with \eqref{eq:hom}, we conclude that the cokernel of $\circ f'$ is $K$.
\end{proof}

\begin{proof}[Proof of Theorem \ref{T:tilting}] We need to show that the endomorphism algebra $\End_{A}(T)$ of the BB-tilting module $T$ is isomorphic to $J(\wtd{Q},\wtd{W};\wtd{C})$.
Recall that the BB-tilting module $T$ is obtained from $\bigoplus_{i\in Q_0} P_i$ by just replacing $P_k$ with $T_k$.
So according to Lemma \ref{L:presentation}, it suffice to check that \begin{enumerate}
\item The map $g$ in Lemma 8.2 and $f'$ in Lemma 8.3 agrees with the map $\pi$;
\item The map $r_1$ in Lemma 8.2 and $g'$ in Lemma 8.3 agrees with the desired relations $R_0,R_1$ and $R_2$ (defined before Lemma \ref{L:res1}).
\end{enumerate}
(1) is clear from the definition of $g,f'$, and $\pi$. 
For (2), we observe that \begin{enumerate}
\item[$\bullet$] The map $r_1$ agrees with $\pi({a^*}^{-1}r)$ for $r \in R_1$ and $\pi(a^*)=g_a$;
\item[$\bullet$] The component map $\partial_{cb}$ in $g'$ agrees with $\pi({b}^{-1}r)$ for $r \in R_0$ and $\pi(b)=b$;
\item[$\bullet$] Similarly, for $b\in C$ the component map $\partial_{acb}$ (resp. $g$) in $g'$ is responsible for the summand $\partial_{ab}W$ (resp. $b^*a^*$) in $R_2$.
\end{enumerate}
\end{proof}

Finally, we prove Corollary \ref{C:tilting}.
\begin{proof}[Proof of Corollary \ref{C:tilting}] Since the cut satisfies \eqref{A:cut}, there is no relation starting from $S_k$. So $S_k$ has projective dimension one, and we have that $0\to P_\out\to P_k\to S_k\to 0$. Hence $\Hom_A(T,P_\out)=\Hom_A(T,P_k)$, and the map $\circ r_1$ in Lemma \ref{L:res1} is in fact injective.
Now $J(Q,W;C)$ has global dimension 2, so $P'$ in \eqref{eq:cone} is zero, and thus the map $\circ g'$ of Lemma \ref{L:res2} is injective.
We conclude that $J(\wtd{Q},\wtd{W};\wtd{C})$ has global dimension 2 as well. The two resolutions of Lemma \ref{L:res1} and \ref{L:res2} also imply that $\{\partial_{c}\wtd{W}\}_{c\in\wtd{C}}$ is a minimal set of generators in $\innerprod{\partial_{\wtd{C}}\wtd{W}}$.
\end{proof}

\section*{Acknowledgement}
The author thanks Mathematical Science Research Institute in Berkeley (MSRI) for its hospitality and support during the research
program Cluster Algebras of Fall 2012 when most of results are obtained. 
He also wants to thank Professor Bernhard Keller for his encouragement.
Finally he thanks the anonymous referee for the careful review and helpful comments.

\bibliographystyle{amsplain}

\end{document}